\documentclass[12pt]{article}

\usepackage{amssymb}%
\usepackage{amsmath}%
\usepackage{amsthm}%

\usepackage{xcolor}%

\allowdisplaybreaks%

{\theoremstyle{plain}%
 \newtheorem{theorem}{Theorem}

}
{\theoremstyle{remark}

}
{\theoremstyle{definition}

}

\begin{document}

\begin{center}
{\Large A further $q$-analogue of a formula due to Guillera}
\end{center}

\begin{center}
{\textsc{John M. Campbell}} 

 \ 

\end{center}

\begin{abstract}
 Hou, Krattenthaler, and Sun have introduced two $q$-analogues of a remarkable series for $\pi^2$ due to Guillera, and these $q$-identities were, 
 respectively, proved with the use of a $q$-analogue of a Wilf--Zeilberger pair provided by Guillera and with the use of ${}_{3}\phi_{2}$-transforms. We 
 prove a $q$-analogue of Guillera's formula for $\pi^2$ that is inequivalent to previously known $q$-analogues of the same formula due to Guillera, including 
 the Hou--Krattenthaler--Sun $q$-identities and a subsequent $q$-identity due to Wei. In contrast to previously known $q$-analogues of Guillera's formula, 
 our new $q$-analogue involves another free parameter apart from the $q$-parameter. Our derivation of this new result relies on the $q$-analogue of 
 Zeilberger's algorithm. 
\end{abstract}



\section{Introduction}
 Let the Pochhammer symbol be denoted so that $(x)_{n} = x(x+1) \cdots (x + n - 1)$, with $(x)_{0} = 1$, and let 
\begin{equation*}
 \left[ \begin{matrix} \alpha, \beta, \ldots, \gamma \vspace{1mm} \\ 
 A, B, \ldots, C \end{matrix} \right]_{n} = \frac{ (\alpha)_{n} (\beta)_{n} 
 	\cdots (\gamma)_{n} }{ (A)_{n} (B)_{n} \cdots (C)_{n}}. 
\end{equation*}
 A remarkable result due to Guillera \cite{Guillera2008} is such that 
\begin{equation}\label{Guilleramain}
 \frac{\pi^2}{4} = \sum_{n=0}^{\infty} \left( \frac{1}{4} \right)^{n} 
 \left[ \begin{matrix} 1, 1, 1 \vspace{1mm} \\ 
 \frac{3}{2}, \frac{3}{2}, \frac{3}{2} \end{matrix} \right]_{n} (3n+2) 
\end{equation}
 and is closely related to Ramanujan's series 
 for $\frac{1}{\pi}$ \cite{Ramanujan1914} of 
 the same convergence rate as in the hypergeometric sum in \eqref{Guilleramain}: 
\begin{equation*}
 \frac{4}{\pi} = \sum_{n=0}^{\infty} \left( \frac{1}{4} \right)^{n} 
 \left[ \begin{matrix} \frac{1}{2}, \frac{1}{2}, \frac{1}{2} \vspace{1mm} \\ 
 1, 1, 1 \end{matrix} \right]_{n} (6n + 1). 
\end{equation*}
 The formula in \eqref{Guilleramain} was proved by 
 Guillera \cite{Guillera2008} via the \emph{Wilf--Zeilberger} (WZ) method \cite{PetkovsekWilfZeilberger1996}, 
 and with the use of the bivariate functions 
 $F(n, k) = 8 B(n, k) n$ and $G(n, k) = B(n, k) (6n + 4k + 1)$, writing 
 $$ B(n, k) = \frac{1}{2^{8n} 2^{4k}} \frac{((2k)!)^2 ((2n)!)^3 }{ ((n+k)!)^2 (k!)^2 (n!)^4}. $$ 
 As a main result in a 2019 contribution to the \emph{Proceedings of the American Mathematical Society} 
 \cite{HouKrattenthalerSun2019}, two $q$-analogues of Guillera's formula in 
 \eqref{Guilleramain} were introduced and proved. 
 We introduce and prove a further $q$-analogue of Guillera's formula 
 that is inequivalent to previously known 
 $q$-analogues of  \eqref{Guilleramain}  \cite{ChenChu2021,HouKrattenthalerSun2019,WangZhong2023,Wei2020}. 

\section{Background and preliminaries}
   The $\pi$ formula in \eqref{Guilleramain}    has been proved in a variety of different ways.     For example,     Chu and Zhang \cite{ChuZhang2014} proved   
   \eqref{Guilleramain} using a series acceleration method     derived via the modified Abel lemma on summation by parts      and Dougall's ${}_{5}F_{4}$-sum.   
  Also, Campbell and Levrie applied a series acceleration method      due to Wilf \cite{Wilf1999}      based on Zeilberger's algorithm  
 \cite[\S6]{PetkovsekWilfZeilberger1996}    to formulate another proof of Guillera's formula.   
  See also \cite{Campbell2022,Chu2021,LevrieCampbell2022}. 
 These past references 
 motivate the development of techniques for deriving $q$-analogues of Guillera's 
 formula, to build on the work of 
 Hou, Krattenthaler, and Sun \cite{HouKrattenthalerSun2019}. 

 The \emph{$q$-shifted factorial} is such that 
\begin{equation}\label{qshift}
 (a;q)_{n} = \prod_{k=0}^{n-1} (1-aq^{k}) 
\end{equation}
 for integers $n \geq 0$ and for complex $q$ such that $|q| < 1$, with $(a;q)_{0} = 1$ by convention. This provides an analogue of the Pochhammer 
 symbol in the sense that $ \lim_{q \to 1} \frac{ \left( q^{x}; q \right)_{n} }{ (1-q)^{n} } = (x)_{n}$. 
 With regard to the notation in 
 \eqref{qshift}, we also write 
 $(a;q)_{\infty} = \lim_{n \to \infty} (a;q)_{n}$, 
 so that 
 $ (a;q)_{n} = \frac{ \left( a;q \right)_{\infty} }{ \left( aq^{n}; q \right)_{\infty} }$. 
 The \emph{$q$-binomial coefficient} may be defined via \eqref{qshift} 
 so that 
 $ \binom{n}{k}_{q} = \frac{(q;q)_n}{(q;q)_k (q;q)_{n-k}}$, 
 with the limiting $q \to 1$ case yielding the binomial coefficient $\binom{n}{k}$. 
 A key to our derivation of a new $q$-analogue of 
 Guillera's formula in \eqref{Guilleramain} 
 is given by the {\tt QDifferenceEquations} package for the Maple Computer Algebra System. 

 Theorem 1.2 from the work of Hou et al.\ \cite{HouKrattenthalerSun2019} 
 gives us that 
\begin{align}
\begin{split}
 & \sum_{n=0}^{\infty} q^{2n(n+1)} (1+ q^{2n+2} - 2 q^{4n+3}) 
 \frac{ \left( q^{2}; q^{2} \right)_{n}^{3} }{ \left( q;q^2 \right)_{n+1}^{3} \left( -1; q \right)_{2n+3} } = \\
 & \frac{1}{2} \sum_{n=0}^{\infty} \frac{q^{2n}}{ (1-q^{2n+1})^{2} } 
\end{split}\label{HKS1}
\end{align}
 and that 
\begin{equation}\label{HKS2}
 \sum_{n=0}^{\infty} q^{\frac{n(n+1)}{2}} \frac{1 - q^{3n+2}}{1-q} 
 \frac{ \left( q;q \right)_{n}^{3} \left( -q;q \right)_{n} }{ \left( q^{3}; q^{2} \right)_{n}^{3} } 
 = (1-q)^{2} \frac{ \left( q^{2}; q^{2} \right)_{\infty}^{4} }{ \left( q;q^{2} \right)_{\infty}^{4} }. 
\end{equation}
 As expressed by Hou et al.\ \cite{HouKrattenthalerSun2019}, 
 by multiplying both sides of \eqref{HKS1} by $(1-q)^2$ and by setting $q \to 1$, 
 this provides an equivalent version of Guillera's formula in \eqref{Guilleramain}. Similarly, 
 setting $q \to 1$ in \eqref{HKS2}, 
 and by applying a property associted with the $q$-Gamma function that gives us that 
 $ \lim_{q \to 1} (1-q) \frac{ \left( q^{2}; q^{2} \right)_{\infty}^{2} }{ \left( q;q^2 \right)_{\infty}^{2} } 
 = \frac{\pi}{2}$, 
 we again obtain an equivalent copy of Guillera's   formula. 

\section{Main result} 
  Instead of using a $q$-analogue $(F_{q}, G_{q})$ of the WZ pair $(F, G)$ employed by Guillera,  as in the work of Hou et al.\    
 \cite{HouKrattenthalerSun2019},   with $\lim_{q \to 1} F_{q}(n, k) = F(n, k)$ and similarly for $G_{q}$,   we apply a $q$-analogue of Zeilberger's algorithm  
  based on the $q$-binomial identity such that  $\sum_{k=0}^{n} q^{k^2} \binom{n}{k}_{q}^{2} = \binom{2n}{n}_{q}$.   This gives us a pair of bivariate 
 $q$-hypergeometric functions inequivalent to that  inolved in \cite{HouKrattenthalerSun2019}.  

 Our main result is given by the $q$-identity highlighted as Theorem \ref{maintheorem} below.  Theorem \ref{maintheorem} can be shown, as we later 
 demonstrate, to provide an equivalent   version of Guillera's formula in \eqref{Guilleramain}, by setting $q \to 1$ and $k = \frac{1}{2}$.  Theorem 
 \ref{maintheorem} is inequivalent to  previously known $q$-analogues of \eqref{Guilleramain}  given in  
 \cite{ChenChu2021,HouKrattenthalerSun2019,WangZhong2023,Wei2020}.   An advantage to our main result, as below, is given by how it involves 
 \emph{another}   free parameter $k$, apart from the $q$-parameter, but this is not the case for  previously known $q$-analogues of Guillera's formula.  

\begin{theorem}\label{maintheorem}
  The identity  
\begin{align*}
 & q \sum_{n=0}^{\infty} 
 \frac{q^{n-2 k+k^2} \left(2 q^k-q^{n+1}-q^{2 n+2}\right)}{1+q^{n+1}} 
 \frac{(q;q)_n^4}{(q;q)_{k-1}^2 (q;q)_{2 n+1} (q;q)_{n-k+1}^2} = \\ 
 & \frac{1}{(q;q)_{\infty }^4} \sum_{n=0}^{\infty} q^{(k+n)^2} \left((q;q)_{\infty }^3-\left(q^{1-n - 
 k};q\right)_{\infty }^2\right) \left(q^{1+n+k};q\right)_{\infty }^2 
\end{align*}
 holds for complex $q$ and $k$ such that $|q| < 1$ and $| k | < 1$. 
\end{theorem}

\begin{proof}
 By setting 
\begin{equation}\label{Fnkproof}
 F(n, k) := \frac{q^{k^2} \binom{n}{k}_q^2}{\binom{2 n}{n}_q}, 
\end{equation}
 an application of the $q$-analogue of Zeilberger's algorithm produces the certificate $$ R(n, k) := \frac{q^{n+3} \left(1-q^k\right)^2 \left(q^{n + 
 1}+q^{2 n+2}-2 q^k\right)}{\left(q^k-q^{n+1}\right)^2 \left(1+q^{n+1}\right) \left(1-q^{2 n+1}\right)}, $$ which leads us to set $G(n, k) := R(n, 
 k) F(n, k)$. We may thus verify that the difference equation 
\begin{equation}\label{verifydifference}
 q^2 F(n+1,k)-q^2 F(n,k) = G(n,k+1)-G(n,k) 
\end{equation}
 holds true. By summing both sides of \eqref{verifydifference} with respect to 
 $n$, a telescoping phenomenon gives us that 
 $$ q^2 F(m+1,k)-q^2 F(0,k)
 = \sum _{n=0}^m (G(n,k+1)-G(n,k)) $$
 for nonnegative integers $m$. 
 From the definition in \eqref{Fnkproof} together with 
 the given constraints on $q$ and $k$, we find, by setting $m \to \infty$, that 
 $$ \frac{q^{k^2+2} \left((q;q)_{\infty }^3-\left(q^{1-k};q\right)_{\infty }^2\right) \left(q^{k+1};q\right)_{\infty }^2}{(q;q)_{\infty }^4}
 = \sum_{n = 0}^\infty (G(n,k+1)-G(n,k)). $$ 
 Writing $H(k)$ in place of the left-hand side of the above equality, a telescoping phenomenon then gives us that 
\begin{equation}\label{270727470767279717074787P7M1A}
 \sum_{n = 0}^{m} H(k+n) = \sum_{n=0}^{\infty} (G(n, k + m + 1) - G(n, k)). 
\end{equation} 
 We find that $\lim_{m \to \infty} G(n, k + m + 1) $ vanishes, 
 again subject to the given constraints on $q$ and $k$. 
 Setting $m \to \infty$ in \eqref{270727470767279717074787P7M1A}, 
 an application of the Dominated Convergence Theorem allows us to interchange the limiting operations 
 resulting on the right-hand side, 
 yielding 
 $ \sum_{n=0}^{\infty} H(k + n) = -\sum_{n=0}^{\infty} G(n, k)$, 
 which is equivalent to the desired result. 
\end{proof}

 Setting $q \to 1$ on both sides of the $q$-identity given in Theorem \ref{maintheorem}, through an application of reindexing and Carlson's theorem to 
 the classical identity $ \frac{1}{(q;q)_{\infty}} = \sum_{n=0}^{\infty} \frac{q^{n^{2}}}{ (q;q)_{n}^{2} } $ 
 derived from the generating function 
 for the sequence enumerating integer partitions of orders $n \geq 0$, 
 we obtain that 
\begin{align}
 & \frac{1}{2} \sum_{n=0}^{\infty} (2 - 2 k + 3 n) 
 \frac{ (1)_n^4 }{ (1)_{k-1}^2 (1)_{2 n+1} (1)_{-k+n+1}^2} = \nonumber \\ 
 & 1 - \sum_{n = 0}^{\infty} \frac{1}{ (1)_{-k-n}^2 (1)_{k+n}^2} = \label{fromgf} \\ 
 & 1-\frac{\psi ^{(1)}(k) \sin ^2(k \pi )}{\pi ^2}, \nonumber 
\end{align}
 where $\psi^{(1)}$ denotes the trigamma function 
 $\psi^{(1)}(z) = \frac{d^{2}}{dz^{2}} \ln \Gamma(z)$, 
 and where this last equality may be verified inductively, by taking the partial sums
 of the series in \eqref{fromgf}. 
 In particular, setting $k = \frac{1}{2}$ gives us that 
 $$ \pi = \sum_{n = 
 0}^{\infty} \frac{(3 n + 2) (1)_n^4}{ (1)_{n + \frac{1}{2}}^2 (1)_{2 n + 1}} $$ 
 and this is equivlaent to Guillera's formula in 
 \eqref{Guilleramain}, according to the Gauss multiplication formula 
 and the relation $(x)_{n} = \frac{\Gamma(x+n)}{\Gamma(x)}$. 

\section{Conclusion}
        Research related to the Hou--Krattenthaler--Sun $q$-analogues of Guillera's formula,  
  as in     \cite{ChenChu2021,HouSun2021,WangZhong2023,Wei2020}, 
    motivates 
     the exploration of further properties of the $q$-analogue in Theorem \ref{maintheorem}.  
    For example, by again setting $q \to 1$, and  
   by applying the operator $\frac{d}{dk} \cdot \big|_{k = \frac{1}{2}}$  
 to both sides of the resultant identity, 
 we obtain 
 $$ \frac{ 7 \zeta (3)}{2} = \sum _{n=0}^{\infty} 
 \left(\frac{1}{4}\right)^n
 \left[ \begin{matrix} 1, 1, 1 \vspace{1mm} \\ 
 \frac{3}{2}, \frac{3}{2}, \frac{3}{2} \end{matrix} \right]_{n} 
 ((6 n+4) O_{n+1} - 1) $$
 where $O_{m} = \frac{1}{1} + \frac{1}{3} + \cdots + \frac{1}{2m-1}$ denotes the $m^{\text{th}}$
 odd harmonic number, and where $\zeta(3) = \frac{1}{1^3} + \frac{1}{2^3} + \cdots $ 
 denotes Ap\'{e}ry's constant. 
 This provides a harmonic sum 
 analogue of Guillera's formula, and we encourage a full exploration of the derivation of such results. 

\subsection*{Acknowledgements}
 The author was supported through a Killam Postdoctoral Fellowship from the Killam Trusts.

 \

{\textsc{John M. Campbell} } 

{\footnotesize Department of Mathematics and Statistics}

{\footnotesize Dalhousie University, Canada} 

{\footnotesize {\tt jmaxwellcampbell@gmail.com}}

\end{document}